\documentclass[1p,10pt]{elsarticle}

\usepackage{amsmath,amssymb,amsfonts,amscd}

\usepackage{latexsym}
\usepackage{amsmath}
\usepackage{amsfonts}
\usepackage{amscd}
\usepackage{epsfig}
\usepackage{amsmath, amssymb, graphicx}
\usepackage{subfig}
\usepackage{graphicx}
\usepackage[all]{xy}

\usepackage{xcolor}
\usepackage{graphics}
\usepackage{amsmath, amssymb, graphics}

\newtheorem{theorem}{Theorem}[section]
\newtheorem{definition}[theorem]{Definition}
\newtheorem{lemma}[theorem]{Lemma}

\newtheorem{example}[theorem]{Example}
\newtheorem{remark}[theorem]{Remark}
\newproof{proof}{Proof}

\numberwithin{equation}{section}

\begin{document}

\begin{frontmatter}







\title{Analytic Fourier-Feynman transforms via the series approximation}


\author[label2]{Hyun Soo Chung}
\ead{hschung@dankook.ac.kr}

\address[label2]{Department of Mathematics,  
                 Dankook University,
                 Cheonan 31116, 
                 Republic of Korea}

\begin{abstract}
In this paper, we first establish an evaluation formula to calculate Wiener integrals of functionals on Wiener space. We then apply our evaluation formula to carry out very easily calculating for the analytic Fourier-Feynman transform of the functionals. Some examples are furnished to illustrate the usefulness of the evaluation formula. Finally, using the evaluation formula, 
 we establish the series approximation  for the analytic Fourier-Feynman transform. 
\end{abstract}

\begin{keyword}
Evaluation formula; unbounded functionals; analytic Fourier-Feynman transform; series approximation.

\medskip
\medskip
 \MSC   28C20,60J65, 46G05.

\end{keyword}

\end{frontmatter}

\tableofcontents




\section{Introduction}\label{sec:intro}

For $T>0$, let $C_0[0,T]$ denote the one-parameter Wiener space, that is, the space
of  all  real-valued continuous functions $x$ on $[0,T]$ with $x(0)=0$.
Let $\mathcal{M}$ denote the class of  all Wiener measurable subsets
of $C_0[0,T]$ and let ${m}$ denote Wiener measure.  
Then, as is well-known,  $(C_0[0,T],\mathcal{M}, {m})$ is 
a complete  measure space. For an integrable functional $F$ on $C_0[0,T]$, the Wiener integral of $F$ is denoted by
$$
\int_{C_0[0,T]} F(x) m(dx).
$$

\par
The concept of the analytic Fourier--Feynman transform(FFT) on the Wiener 
space $C_0[0,T]$, initiated by Brue \cite{Brue}, has been developed in 
the literature. This transform  and its properties are similar in many 
respects to the ordinary Fourier  transform. For an elementary 
introduction to the analytic FFT, see \cite{SS04} 
and the references cited therein. Also we refer to references \cite{CS80,CCKSY05,Cho13,CPS93,HPS96,HPS97-1,HPS97-2}. 
Many mathematicians have been studied the analytic FFT of various functionals on Wiener space. One of topics to the theory of Fourier-Feynman transform is concerned with the classes of all polynomial functionals. For instance, see \cite{HGN2011,N96}. These classes have been used to explain certain physical phenomena.

However, there are some difficulties to evaluate  analytic FFT for high order polynomial functionals as follows; Let $\langle v,x\rangle$ denote the Paley-Wiener-Zygmund(PWZ) stochastic integral \cite{PWZ33}. 
For each $n=1,2,\cdots$, let $G_n(x)=\langle v,x\rangle^n$ with $\|v\|_2=1$. To calculate the analytic FFT of $G_n$, we have to consider following Wiener integral
\begin{equation} \label{eq:ob1}
\int_{C_0[0,T]} [\langle v,x\rangle+\langle v,y\rangle]^n m(dx).
\end{equation}
One can see that it is not easy to calculate of the Wiener integral \eqref{eq:ob1} because the Lebesgue integral
\begin{equation} \label{eq:ob2}
\biggl(\frac{1}{2\pi  }\biggr)^{\frac12} \int_{\mathbb{R}} (u+\langle v,y\rangle)^n \exp\biggl\{-\frac{u^2}{2}\biggr\}du 
\end{equation}
is appeared in the calculation of the Wiener integral \eqref{eq:ob1} whenever we will apply the change of variable theorem. In order to evaluate the Lebesgue integral \eqref{eq:ob2}, we can use the integration by parts formulas $n$-times. However it is very difficult and complicate.

In this paper, we establish a new evaluation formula  to figure out these difficulties and complications. Using the evaluation formula, we obtain various examples involving the FFTs very easily. Finally, we give the series approximation for the   analytic FFT.

\section{Definitions and preliminaries}

We first list key some definitions and preliminaries which are needed to understand this paper.

A subset $B$ of $C_0[0,T]$ is said to be  scale-invariant measurable
provided $\rho B\in \mathcal{M}$ for all $\rho>0$, and a  scale-invariant
measurable set $N$ is said to be scale-invariant null provided $m(\rho N)=0$ 
for all $\rho>0$. A property that holds except on  a scale-invariant  null set 
is said to be hold  scale-invariant almost everywhere(s-a.e.). If two functionals 
$F$ and $G$ are equal s-a.e.,  we write $F\approx G$.

For  $v\in L_2[0,T]$ and $x\in C_0[0,T]$, let $\langle{v,x\rangle}$   denote the PWZ stochastic integral.
It is well-known fact that for each $v\in L_2[0,T]$, $\langle{v,x\rangle}$ exists for a.e. $x\in C_0[0,T]$ and if $v\in L_2[0,T]$ is a function of bounded variation, $\langle{v,x\rangle}$ equals to the Riemann-Stieltjes integral $\int_0^T v(t) dx(t)$ for s-a.e. $x\in C_0[0,T]$. Also,   $\langle{v,x\rangle}$  has the expected linearity property. Furthermore, $\langle{v,x\rangle}$ is a Gaussian random variable with mean $0$ and variance $\|v\|_2^2$.
For a more detailed study of the PWZ stochastic integral, see  \cite{R3,CLC14,CPS93,HPS95,PWZ33,SS04}.

We are ready to recall the definitions of analytic Feynman integral and analytic FFT on Wiener space.  

 Let $\mathbb C$, $\mathbb C_+$ and $\mathbb{\widetilde C}_+$ denote the set of  
complex numbers, complex numbers with positive real part   and  nonzero complex 
numbers with nonnegative real part, respectively. For each $\lambda \in \mathbb C$,
$\lambda^{1/2}$ denotes the principal square root of $\lambda$; i.e., $\lambda^{1/2}$ 
is always chosen to have positive real part, so that  
$\lambda^{-1/2}=(\lambda^{-1})^{1/2}$ is  in $\mathbb C_+$ for 
all $\lambda\in\widetilde{\mathbb C}_+$.  
Let  $F$ be 
a $\mathbb C$-valued scale-invariant measurable functional on $C_0[0,T]$ 
such that 
$$
J(\lambda)\equiv\int_{C_0[0,T]} F(\lambda^{-1/2} x )m(dx) 
$$
exists as a finite number for all $\lambda>0$. If there exists a function 
$J^* (\lambda)$ analytic on $\mathbb C_+$ such that  
$J^*(\lambda)=J(\lambda)$ for all $\lambda>0$, then  $J^*(\lambda)$ 
is defined to be the   analytic  Wiener integral  of $F$  over $C_0[0,T]$ with parameter $\lambda$, 
and for $\lambda \in \mathbb C_+$ we write
$$
J^*(\lambda)=\int_{C_0[0,T]}^{\text{\rm anw}_{\lambda}} 
F(x)m(dx).
$$
Let $q$ be a nonzero real number and let $F$ be a functional such that
$\int_{C_0[0,T]}^{\text{\rm anw}_{\lambda}} 
F(x) m(dx)$ 
exists for all $\lambda \in \mathbb C_+$. If the following limit exists, we call
it the   analytic Feynman integral of $F$ with parameter $q$ and we 
write
\[
\int_{C_0[0,T]}^{\text{\rm anf}_{q}} 
F(x) m(dx)
= \lim_{\begin{subarray}{1}
\lambda\to -iq \\ \lambda\in \mathbb C_+\end{subarray}}
\int_{C_0[0,T]}^{\text{\rm anw}_{\lambda}} 
F(x) m(dx).
\]

Next (see \cite{CCKSY05,HPS97-2}) we  state the definition of the  analytic FFT.

\begin{definition} 
 For $\lambda\in\mathbb{C}_+$ 
and $y \in C_{0}[0,T]$, let
$$
T_{\lambda}(F)(y)
=\int_{C_0[0,T]}^{\text{\rm anw}_{\lambda}} 
F(y+x) m(dx).
$$
We define the $L_1$ analytic Fourier-Feynman transform, $T_{q}^{(1)}(F)$ of $F$, 
by the formula  
\[
T_{q }^{(1)}(F)(y)
= \lim_{\begin{subarray}{1}
\lambda\to -iq \\  \lambda\in \mathbb C_+\end{subarray}}
T_{\lambda} (F)(y)
\]
for s-a.e. $y\in C_0[0,T]$.
\end{definition}

\par
We note that   $T_{q}^{(1)}(F)$ is  defined only s-a.e..
We also note that if $T_{q}^{(1)}(F)$  exists  and if $F\approx G$, then 
$T_{q}^{(1)}(G)$ exists  and  $T_{q}^{(1)}(G)\approx T_{q}^{(1)}(F)$.

 The following Wiener integration formula is used several times in this paper.  
Let $\{\alpha_1,\alpha_2,\cdots\}$ be any complete orthonormal set of functions in   $L_2[0,T]$ and let $h:\mathbb {R}^{n} \rightarrow \mathbb {R}$ be Lebesgue measurable.     Then
 \begin{equation}  \label{eq:well77}
\aligned
 \int_{C_0[0,T]}h(\langle{\alpha_1,x\rangle}, &\cdots,\langle{\alpha_n,x\rangle}) m(dx)\\
&=\biggl(\frac{1}{\sqrt{2\pi  }}\biggr)^{n}
\int_{\Bbb R^n}h(\vec{u})
  \exp \bigg\{ - \sum\limits_{j=1}^n\frac{u_j^2}{2 } \bigg\} d\vec{u}
\endaligned
\end{equation} 
in the sense that if either side of \eqref{eq:well77}  exists, both sides exist and equality holds.

We finish this section by giving the functionals on Wiener space which are used in this paper. Let $\{\alpha_1,\cdots,\alpha_n\}$ be a complete orthonormal set in $L_2[0,T]$ and let $F$ be a functional defined by the formula 
\begin{equation} \label{eq:functional}
F(x)= \langle \alpha_1,x \rangle^{2p_1} \times \cdots \times \langle \alpha_1,x \rangle^{2p_n}=\prod_{j=1}^n \langle \alpha_j,x \rangle^{2p_j}
\end{equation}
where $p_1,p_2,\cdots,p_{n-1}$ and $p_n$ are nonnegative integers.
Then one can see that the functionals defined in equation \eqref{eq:functional} are unbounded functionals, see \cite{HGN2011, N96}.  

\begin{remark}
Let $\mathcal{P}$ be the set of all functionals of the form
\begin{equation*} \label{eq:12}
H(x)=h(\langle \alpha_1,x \rangle,\cdots,\langle \alpha_n,x \rangle)
\end{equation*}
where $h$ is a continuous function on $\mathbb{R}^n$. By the Bolzano-Weierstrass theorem, there is a sequence $\{f_n\}$ of polynomial functions such that $\|h-f_n\|_\infty=\sup_{\vec u \in \mathbb{R}^n} |h(\vec u)-f_n(\vec u)| \rightarrow 0$ as $n \rightarrow \infty$.   Thus, the polynomial functionals such as equation \eqref{eq:12} are meaningful objects to study the FFTs. The usefulness of the functionals \eqref{eq:functional} will be explained in Section $5$ below.
 \end{remark}

\section{An evaluation formula}

In this section, we give an evaluation formula for the Wiener integrals. To do this, we shall start by giving two lemmas.  The first lemma is the formula for the Lebesgue integral.

\begin{lemma}
Let $s$ be a nonnegative integer. Then we have
\begin{equation} \label{eq:int1}
\aligned
&\int_{\mathbb{R}} u^s  \exp \biggl\{-  \frac{u^2}{2}\biggr\}du=2^{\frac{s-1}{2}} (1+(-1)^s) \Gamma\biggl(\frac{1+s}{2}\biggr)
\endaligned
\end{equation} 
where $\Gamma$ is the gamma function defined by the formula
$$
\Gamma(r)=\int_0^\infty t^{r-1}e^{-r}dt
$$
for a complex number $r$ with $Re(r)>0$.
\end{lemma}

We now state some properties of the Gamma function $\Gamma$. 
For any positive integer $n$, let $n!=n\times(n-1)\times(n-2)\times\cdots\times1$, and let $(2n-1)!!=(2n-1)\times(2n-3)\times(2n-5)\times\cdots \times 3\times 1$ and set $(-1)!!=1$. Then  
\begin{enumerate}
  \item [(i)] $\Gamma(n)=(n-1)!$ for all positive integer $n$.
  \item [(ii)] $\Gamma(s+1)=s\Gamma(1)$ for all positive real number $s$.
  \item [(iii)] $\Gamma(n+\frac12)=\frac{(2n-1)!!}{2^n} \sqrt{\pi}$ for all positive integer $n$.
\end{enumerate}

In our next lemma, we establish a Wiener integration formula. 

\begin{lemma} \label{lem:0}
Let $p$ be a nonnegative integer and
let $\alpha$ be an element of $L_2[0,T]$ with $\|\alpha\|_2=1$.  Then for all nonzero real numbers $\gamma$ and $\beta$, we have
\begin{equation} \label{eq:opi}
\aligned
&\int_{C_0[0,T]}  [\gamma\langle \alpha,x \rangle+\beta\langle \alpha,y \rangle]^{2p} m(dx)\\
&\qquad\qquad =   \sum\limits_{s=0}^{p} {}_{2p} C_{2s} {(2s-1)!!} \gamma^{2s} \beta^{2p-2s} \langle \alpha,y \rangle^{2p-2s}   
\endaligned
\end{equation}
for $y\in C_0[0,T]$, where  ${}_n C_k =\frac{n!}{k!(n-k)!}$ for nonnegative integers $n$ and $k$ with $n \geq k$.
\end{lemma}
\begin{proof}
For $y\in C_0[0,T]$, let $v= \langle \alpha,y \rangle$. Then for all nonzero real numbers $\gamma$ and $\beta$, and  $y\in C_0[0,T]$, using equation \eqref{eq:well77} we have
$$
\aligned
&\int_{C_0[0,T]} [\gamma\langle \alpha,x \rangle+\beta\langle \alpha,y \rangle]^{2p} m(dx)
=  \biggl(\frac{1}{2\pi}\biggr)^{\frac{1}{2}} \int_{\mathbb{R}} [\gamma u+\beta v]^{2p}\exp \biggl\{-  \frac{u^2}{2}\biggr\}du.
\endaligned
$$
We now state the binomial formula as follows. For $a,b,u,v \in \mathbb{R}$ and for all $n=1,2,\cdots$, we have
$$
(au+bv)^n =\sum\limits_{k=0}^n {}_n C_k (au)^k (bv)^{n-k}=\sum\limits_{k=0}^n {}_n C_k a^k b^{n-k} u^k  v^{n-k}.
$$
Using the binomial formula,   equation \eqref{eq:int1} and some properties of the Gamma function, we have
\begin{equation} \label{eq:456}
\aligned
&\int_{C_0[0,T]} [\gamma\langle \alpha,x \rangle+\beta\langle \alpha,y \rangle]^{2p} m(dx)\\
&= \biggl(\frac{1}{2\pi}\biggr)^{\frac{1}{2}}\sum\limits_{k=0}^{2p} {}_{2p} C_k \gamma^k \beta^{2p-k}v^{2p-k} \int_{\mathbb{R}} u^k  \exp \biggl\{-  \frac{u^2}{2}\biggr\}du \\
&= \biggl(\frac{1}{2\pi}\biggr)^{\frac{1}{2}}\sum\limits_{k=0}^{2p} {}_{2p} C_k \gamma^k \beta^{2p-k}v^{2p-k} 2^{\frac{k-1}{2}} (1+(-1)^{k}) \Gamma\biggl(\frac{1+k}{2}\biggr)  \\
&= \biggl(\frac{1}{2\pi}\biggr)^{\frac{1}{2}}\sum\limits_{s=0}^{p} {}_{2p} C_{2s} \gamma^{2s} \beta^{2p-2s} v^{2p-2s} 2^{\frac{2s-1}{2}} (1+(-1)^{2s}) \Gamma\biggl(\frac{1+2s}{2}\biggr)  \\
&= \biggl(\frac{1}{2\pi}\biggr)^{\frac{1}{2}}\sum\limits_{s=0}^{p} {}_{2p} C_{2s} \gamma^{2s} \beta^{2p-2s}v^{2p-2s} 2^{s+\frac12}    \Gamma\biggl( {s+\frac12} \biggr)   \\
&= \biggl(\frac{1}{\pi}\biggr)^{\frac{1}{2}}\sum\limits_{s=0}^{p} {}_{2p} C_{2s} \gamma^{2s} \beta^{2p-2s}v^{2p-2s} 2^{s}  \frac{(2s-1)!!}{2^s} \sqrt{\pi}    \\
&=  \sum\limits_{s=0}^{p} {}_{2p} C_{2s}{(2s-1)!!} \gamma^{2s} \beta^{2p-2s} v^{2p-2s},
\endaligned
\end{equation}
  which completes the proof of the lemma as desired.
\end{proof}

Using equation \eqref{eq:opi} in Lemma \ref{lem:0}, we can establish the evaluation formula for the Wiener integral.

\begin{theorem} \label{theorem:1}
 Let $F$ be as in equation \eqref{eq:functional} above. Then for all nonzero real numbers $\gamma$ and $\beta$, we have
\begin{equation} \label{eq:11}
\aligned
 &\int_{C_0[0,T]}F(\gamma x+\beta y) m(dx)\\
&\qquad=\prod_{j=1}^n \biggl[ \sum\limits_{s=0}^{p_j} {}_{2p_j} C_{2s}  {(2s-1)!!} \gamma^{2s} \beta^{2p_j-2s} \langle \alpha_j,y\rangle^{2p_j-2s}   \biggr]
\endaligned
\end{equation}
for $y\in C_0[0,T]$. 
\end{theorem}
\begin{proof}
We first note that for each $j=1,2,\cdots,n$, let $v_j= \langle \alpha_j,y \rangle$. Then for all nonzero real numbers $\gamma$ and $\beta$, and  $y\in C_0[0,T]$,
$$
\aligned
&\int_{C_0[0,T]} F(\gamma x+\beta y) m(dx)\\
&=\int_{C_0[0,T]} \prod_{j=1}^n [\gamma\langle \alpha_j,x \rangle+\beta\langle \alpha_j,y \rangle]^{2p_j} m(dx)\\
&= \biggl(\frac{1}{2\pi}\biggr)^{\frac{n}{2}} \int_{\mathbb{R}^n} \prod_{j=1}^n [\gamma u_j+\beta v_j]^{2p_j}\exp \biggl\{-\sum\limits_{j=1}^n \frac{u_j^2}{2}\biggr\}d\vec{u}\\
&= \prod_{j=1}^n \biggl[\biggl(\frac{1}{2\pi}\biggr)^{\frac{1}{2}} \int_{\mathbb{R}} [\gamma u_j+\beta v_j]^{2p_j}\exp \biggl\{-  \frac{u_j^2}{2}\biggr\}du_j\biggr]\\ 
& =\prod_{j=1}^n \biggl[\int_{C_0[0,T]} [\gamma\langle \alpha_j,x \rangle+\beta \langle \alpha_j,y \rangle]^{2p_j} m(dx)\biggr].
\endaligned
$$
Finally, using equation \eqref{eq:opi}   $n$-times repeatedly, we can establish equation \eqref{eq:11} as desired.
 \end{proof}

\section{Analytic FFT via the evaluation formula}

In this section we give an application of our evaluation formula.
 Theorem \ref{thm:1} is one of the main results in this paper.  

\begin{theorem} \label{thm:1}
Let $F$   be as in Theorem \ref{theorem:1} above and let $q$ be a nonzero real number. Then the 
analytic FFT $T_q^{(1)}(F)$ of $F$ exists and is given by the formula
\begin{equation} \label{eq:for22}
T_q^{(1)}(F)(y)=\prod_{j=1}^n \biggl[ \sum\limits_{s=0}^{p_j} {}_{2p_j} C_{2s} {(2s-1)!!}\biggl(\frac{i}{q}\biggr)^s \langle \alpha_j,y \rangle^{2p_j-2s}        \biggr].
\end{equation}
\end{theorem}
\begin{proof}
In equation \eqref{eq:11}, set $\gamma=\lambda^{-\frac12}$ and $\beta=1$ for $\lambda>0$. Then it follows  that
 that for all $\lambda>0$ and s-a.e. $y\in C_0[0,T]$,  we have
\begin{equation}  \label{eq:111}
T_\lambda(F)(y)=\prod_{j=1}^n \biggl[ \sum\limits_{s=0}^{p_j} {}_{2p_j} C_{2s} {(2s-1)!!}\lambda^{-s} \langle \alpha_j,y \rangle^{2p_j-2s}       \biggr].
\end{equation}
From this we observe that $T_\lambda(F)(y)$ of  $F$ exists for all $\lambda>0$. We will show that the analytic FFT $T_q^{(1)}(y)$ of $F$ exists. To do this, for $\lambda \in \mathbb{C}_+$, let
$$
J^*(\lambda)= \prod_{j=1}^n \biggl[ \sum\limits_{s=0}^{p_j} {}_{2p_j} C_{2s}{(2s-1)!!} \lambda^{-s} \langle \alpha_j,y \rangle^{2p_j-2s}        \biggr].
$$
Then $J(\lambda)=J^*(\lambda)$ for all $\lambda$.
Let $\Lambda$ be any simple closed contour in $\mathbb{C}_+$. Then  using   the Cauchy theorem, we have
$$
\aligned
& \int_{\Lambda} J^*(\lambda) d\lambda= \int_{\Lambda}   \prod_{j=1}^n \biggl[ \sum\limits_{s=0}^{p_j} {}_{2p_j} C_{2s}  {(2s-1)!!} \lambda^{-s} \langle \alpha_j,y \rangle^{2p_j-2s}       \biggr] d\lambda=0
\endaligned
$$
because the function $ \sum\limits_{s=0}^{p_j} {}_{2p_j} C_{2s}{(2s-1)!!} \lambda^{-s} \langle \alpha_j,y \rangle^{2p_j-2s}       $ is an analytic function of $\lambda$ in $\mathbb{C}_+$. Hence using   Morera's theorem, we conclude that $J^*(\lambda)$ is analytic on $\mathbb{C}_+$.  It remain to  show that
$$
\lim_{\begin{subarray}{1}
\lambda\to -iq \\  \lambda\in \mathbb C_+\end{subarray}}J^*(\lambda)=\prod_{j=1}^n \biggl[ \sum\limits_{s=0}^{p_j} {}_{2p_j} C_{2s}{(2s-1)!!} \biggl(\frac{i}{q}\biggr)^s \langle \alpha_j,y \rangle^{2p_j-2s}     \biggr].  
$$
However, it is immediate consequence from the fact that the functions
$\lambda^s$, $s=1,2,\cdots$, are continuous and analytic on $\mathbb
C_+$. Thus  we complete the proof of Theorem \ref{thm:1} as desired.
\end{proof}

We now give some formulas for the analytic FFT via the evaluation formula obtained by equation \eqref{eq:for22}. We first give several formulas for the $1$-dimensional functionals as a table.

 \begin{table}[ht]
\begin{center}
\begin{tabular}{ |l| l  |l|    }
\hline
   $n=1, p_j=j$
   &  $   \hbox{analytic FFT of} \,\, F_j, j=1,2,3,4  $    \\\hline
   $F_1(x)=\langle \alpha_1,x \rangle^2$  
   &   $\langle \alpha_1,y \rangle^2+\frac{i}{q}$
 
$$
                                                                                                \\\hline
 $F_2(x)=\langle \alpha_1,x \rangle^4$ 
   &  $
  \langle \alpha_1,y \rangle^4+\frac{6i}{q}\langle \alpha_1,y \rangle^2-\frac{3}{q^2}$
 
$$
                                                                                                \\\hline
 $F_3(x)=\langle \alpha_1,x \rangle^6$  
  &  $
  \langle \alpha_1,y \rangle^6+\frac{15i}{q}\langle \alpha_1,y \rangle^4-\frac{45}{q^2}\langle \alpha_1,y \rangle^2-\frac{15i}{q^3}$
 
$$
                                                                                                \\\hline
$F_4(x)=\langle \alpha_1,x \rangle^4$  
  &  $
  \langle \alpha_1,y \rangle^8+\frac{28i}{q}\langle \alpha_1,y \rangle^6-\frac{210}{q^2}\langle \alpha_1,y \rangle^4
-\frac{320i}{q^3}\langle \alpha_1,y \rangle^2+\frac{105}{q^4}$
 
$$
                                                                                                \\\hline
  
\end{tabular}
\end{center}
\medskip
\caption{Formulas for the $1$-dimensional functionals}
\end{table}

From now on, we next give a  formula for the analytic FFT with the multi-dimensional functionals.


\begin{example} \label{ex:2}
Let $F_5(x)=\langle \alpha_1,x \rangle^2\langle \alpha_2,x \rangle^4$ (Set $n=2; p_1=1, p_2=2$ in equation \eqref{eq:functional}). Then for s-a.e. $y\in C_0[0,T]$, we have
$$
\aligned
T_q^{(1)}(F_5)(y) 
&=\prod_{j=1}^2 \biggl[ \sum\limits_{s=0}^{p_j} {}_{2p_j} C_{2s} \biggl(\frac{i}{q}\biggr)^s \langle \alpha_j,y \rangle^{2p_j-2s}     {(2s-1)!!}  \biggr]\\
&=\biggl[\sum\limits_{s=0}^1 {}_2C_{2s}{(2s-1)!!}\biggl(\frac{i}{q}\biggr)^s \langle \alpha_1,y \rangle^{2-2s}    \biggr]\\
&\qquad\times \biggl[\sum\limits_{s=0}^2 {}_4C_{2s} {(2s-1)!!}\biggl(\frac{i}{q}\biggr)^s \langle \alpha_2,y \rangle^{4-2s}     \biggr]\\
&=\biggl[ \langle \alpha_1,y \rangle^2+\frac{i}{q}\biggr]\biggl[ \langle \alpha_2,y \rangle^4+\frac{6i}{q}\langle \alpha_2,y \rangle^2-\frac{3}{q^2}\biggr]
.\endaligned
$$
\end{example}

\begin{remark}
From the definition of analytic FFT, one can observe that for $\lambda>0$,
$$
\aligned
T_\lambda(F_5)(y)&=\int_{C_0[0,T]}    [\lambda^{-\frac12}\langle \alpha_1,x \rangle+\langle \alpha_1,y \rangle]^{2}[\lambda^{-\frac12}\langle \alpha_2,x \rangle+\langle \alpha_2,y \rangle]^{4} m(dx)\\
&=\int_{C_0[0,T]}    [\lambda^{-1}\langle \alpha_1,x \rangle^2+2\lambda^{-\frac12}\langle \alpha_1,x \rangle\langle \alpha_1,y \rangle+\langle \alpha_1,y \rangle^2] \\
&\qquad \times  [\lambda^{-2}\langle \alpha_2,x \rangle^4+4\lambda^{-\frac32} \langle \alpha_2,x \rangle^3\langle \alpha_2,y \rangle
+6 \lambda^{-1}\langle \alpha_2,x \rangle^2 \langle \alpha_2,y \rangle^2\\
&\qquad\qquad+4 \lambda^{-\frac12}\langle \alpha_2,x \rangle\langle \alpha_2,y \rangle^3
 +\langle \alpha_2,y \rangle^4] m(dx)\\
&=\int_{C_0[0,T]}    [\lambda^{-3}\langle \alpha_1,x \rangle^2\langle \alpha_2,x \rangle^4+  6\lambda^{-2} \langle \alpha_1,x \rangle^2\langle \alpha_2,x \rangle^2 \langle \alpha_2,y \rangle^2\\
&\qquad\qquad\qquad+\lambda^{-1} \langle \alpha_1,x \rangle^2 \langle \alpha_2,y \rangle^4+\lambda^{-2} \langle \alpha_2,x \rangle^4 \langle \alpha_1,y \rangle^2\\
&\qquad\qquad\qquad\qquad+ 6 \lambda^{-1} \langle \alpha_2,x \rangle^2 \langle \alpha_1,y \rangle^2 \langle \alpha_2,y \rangle^2 + \langle \alpha_1,y \rangle^2 \langle \alpha_2,y \rangle^4 ]  m(dx)\\
&= 3\lambda^{-3} +  6\lambda^{-2} \langle    \alpha_2,y \rangle^2 +\lambda^{-1}   \langle \alpha_2,y \rangle^4+3 \lambda^{-2}   \langle \alpha_1,y \rangle^2\\
&\qquad\qquad\qquad\qquad+ 6 \lambda^{-1}   \langle \alpha_1,y \rangle^2 \langle \alpha_2,y \rangle^2 + \langle \alpha_1,y \rangle^2 \langle \alpha_2,y \rangle^4.   
 \endaligned
$$
It can be analytically
continued on $\mathbb C_+$ and so letting $\lambda \rightarrow -iq$, we have
$$
\aligned
&T_q^{(1)}(F_5)(y)=\biggl[ \langle \alpha_1,y \rangle^2+\frac{i}{q}\biggr]\biggl[ \langle \alpha_2,y \rangle^4+\frac{6i}{q}\langle \alpha_2,y \rangle^2-\frac{3}{q^2}\biggr]
.
\endaligned
$$
As seen above, it is not easy to calculate. Hence our evaluation formula is a meaningful subject.
\end{remark}

We give more explicit formulas for the analytic FFT with the multi-dimensional functionals.

\begin{example} \label{ex:3}
Let $F_6(x)=\langle \alpha_1,x \rangle^4\langle \alpha_2,x \rangle^2\langle \alpha_3,x \rangle^6$ (Set $n=3; p_1=2, p_2=1, p_3=3$ in equation \eqref{eq:functional}). Then for s-a.e. $y\in C_0[0,T]$, we have
$$
\aligned
T_q^{(1)}(F_6)(y)
&=\prod_{j=1}^3\biggl[ \sum\limits_{s=0}^{p_j} {}_{2p_j} C_{2s} {(2s-1)!!}\biggl(\frac{i}{q}\biggr)^s \langle \alpha_j,y \rangle^{2p_j-2s}       \biggr]\\
&=\biggl[\sum\limits_{s=0}^2 {}_4C_{2s} {(2s-1)!!}\biggl(\frac{i}{q}\biggr)^s \langle \alpha_1,y \rangle^{4-2s}     \biggr]\\
&\qquad\times \biggl[\sum\limits_{s=0}^1 {}_2C_{2s} {(2s-1)!!}\biggl(\frac{i}{q}\biggr)^s \langle \alpha_2,y \rangle^{2-2s}      \biggr]\\
&\qquad\qquad\times \biggl[\sum\limits_{s=0}^3 {}_6C_{2s}{(2s-1)!!}\biggl(\frac{i}{q}\biggr)^s \langle \alpha_3,y \rangle^{6-2s}      \biggr]\\
&=\biggl[ \langle \alpha_1,y \rangle^4+\frac{6i}{q}\langle \alpha_1,y \rangle^2-\frac{3}{q^2}\biggr]\biggl[ \langle \alpha_2,y \rangle^2+\frac{i}{q}\biggr]\\
&\qquad \times\biggl[ \langle \alpha_3,y \rangle^6+\frac{15i}{q}\langle \alpha_3,y \rangle^4-\frac{45}{q^2}\langle \alpha_3,y \rangle^2-\frac{15i}{q^3}\biggr].
\endaligned
$$
\end{example}

\begin{example} \label{ex:4}
Let $F_7(x)=\langle \alpha_1,x \rangle^4\langle \alpha_3,x \rangle^2\langle \alpha_4,x \rangle^4$ (Set $n=4; p_1=2, p_2=0, p_3=1, p_4=3$ in equation \eqref{eq:functional}). Then for s-a.e. $y\in C_0[0,T]$, we have
$$
\aligned
T_q^{(1)}(F_7)(y)
&=\prod_{j=1}^4\biggl[ \sum\limits_{s=0}^{p_j} {}_{2p_j} C_{2s} {(2s-1)!!}\biggl(\frac{i}{q}\biggr)^s \langle \alpha_j,y \rangle^{2p_j-2s}        \biggr]\\
&=\biggl[\sum\limits_{s=0}^2 {}_4C_{2s}{(2s-1)!!}\biggl(\frac{i}{q}\biggr)^s \langle \alpha_1,y \rangle^{4-2s}     \biggr]\\
&\qquad\times\biggl[\sum\limits_{s=0}^0 {}_0C_{2s} {(2s-1)!!}\biggl(\frac{i}{q}\biggr)^s \langle \alpha_2,y \rangle^{-2s}      \biggr]\\
&\qquad\qquad\times \biggl[\sum\limits_{s=0}^1 {}_2C_{2s}{(2s-1)!!}\biggl(\frac{i}{q}\biggr)^s \langle \alpha_3,y \rangle^{2-2s}      \biggr]\\
&\qquad\qquad\qquad\times \biggl[\sum\limits_{s=0}^2 {}_4C_{2s}{(2s-1)!!}\biggl(\frac{i}{q}\biggr)^s \langle \alpha_4,y \rangle^{6-2s}      \biggr]\\
&=\biggl[ \langle \alpha_1,y \rangle^4+\frac{6i}{q}\langle \alpha_1,y \rangle^2-\frac{3}{q^2}\biggr]\\
&\qquad \times \biggl[ \langle \alpha_3,y \rangle^2+\frac{i}{q}\biggr]
\biggl[ \langle \alpha_4,y \rangle^4+\frac{6i}{q}\langle \alpha_1,y \rangle^2-\frac{3}{q^2}\biggr].
\endaligned
$$
\end{example}

\begin{remark}
We considered only seven  functionals. But, we can obtain various functionals with high dimensional.
\end{remark}

\section{Series approximation}

  Using equation \eqref{eq:for22}, in this section, we shall establish a series
approximation for the analytic FFT through several steps."
 
{\bf  Step $1$ :}
Let $h\in C^\infty(\mathbb{R}^n)$ such that
\begin{equation} \label{eq:concon}
\int_{\mathbb{R}^n} |h(\vec{u})| \exp\biggl\{-a\sum\limits_{j=1}^n u_j^2 \biggr\} d\vec{u}<\infty
\end{equation}
for all $a>0$. Then the Maclaurin series expansion of $h$ is  give by the formula
\begin{equation} \label{eq:ap111}
h(\vec{u})= h(\vec{0})+ \sum\limits_{k=1}^\infty \frac{1}{k!}\biggl(\sum\limits_{i_1,\cdots,i_k=1}^n h_{u_{i_1}\cdots u_{i_k}}(\vec{0}) u_{i_1}\cdots u_{i_k}\biggr)
\end{equation}
where $h_{u_{i_1}\cdots u_{i_k}}$ is $k$-th the derivative of $h$. Assume that $h_{u_{i_1}\cdots u_{i_k}}(\vec{0})=1$ for all derivatives of $h$ (In fact, all process of this development can be applied in the case that any $k$-th
order partial derivatives $h_{u_{i_1} \cdots u_{i_k}}$'s have same value
at $\vec 0$ is constant). Then equation \eqref{eq:ap111} can be written by
\begin{equation} \label{eq:ap1}
\aligned
h(\vec{u})&=  h(\vec{0})+ \sum\limits_{k=1}^\infty \frac{1}{k!} (u_1+\cdots+u_n)^k \\
&\equiv \lim_{r \rightarrow \infty} h_r(\vec{u})
\endaligned
\end{equation}
where $h_r(\vec{u})=h(\vec{0})+ \sum\limits_{k=1}^r \frac{1}{k!} (u_1+\cdots+u_n)^k, r=1,2,\cdots$. Hence we have
$
|h(\vec{u})-h_r(\vec{u})|\rightarrow 0
$
as $r\rightarrow \infty$.

{\bf  Step $2$ :}
For each $r=1,2,\cdots$, let $H_r(x)=h_r(\langle \alpha_1,x\rangle,\cdots,\langle \alpha_n,x\rangle)$ and let $H(x)=h(\langle \alpha_1,x\rangle,\cdots,\langle \alpha_n,x\rangle)$. Then one can check that for all $\rho>0$,
$$
\int_{C_0[0,T]} |H(\rho x) - H_r(\rho x)| m(dx) \rightarrow 0
$$ 
as $r\rightarrow\infty$ because for all $\rho>0$, we see that
$$
\aligned
L_r&\equiv\int_{C_0[0,T]}|H(\rho x)-H_r(\rho x)| m(dx)\\
&=\int_{C_0[0,T]}|h(\rho\langle \alpha_1,x\rangle,\cdots,\rho\langle  \alpha_n,x\rangle)-h_r(\rho\langle \alpha_1,x\rangle,\cdots,\rho\langle \alpha_n,x\rangle)| m(dx)\\
&=\biggl(\frac{1}{\sqrt{2\pi \rho^2}}\biggr)^{n} \int_{\mathbb{R}^n} |h(\vec{u})-h_r(\vec{u})|\exp\biggl\{ -\sum\limits_{j=1}^n\frac{u_j^2}{2\rho^2}\biggr\} d\vec{u}\\
&\leq\biggl(\frac{1}{\sqrt{2\pi \rho^2}}\biggr)^{n}\int_{\mathbb{R}} |h(\vec{u})| \exp\biggl\{ -\sum\limits_{j=1}^n\frac{u_j^2}{2\rho^2}\biggr\} d\vec{u}\\
&\qquad\qquad+\biggl(\frac{1}{\sqrt{2\pi\rho^2 }}\biggr)^{n}\int_{\mathbb{R}} |h_r(\vec{u})|\exp\biggl\{ -\sum\limits_{j=1}^n\frac{u_j^2}{2\rho^2}\biggr\} d\vec{u}<\infty
\endaligned
$$
for all $r=1,2,\cdots$.
Hence we can conclude that 
$L_r$
tends to zero as $r \rightarrow 0$
from the dominated convergence theorem.

{\bf  Step $3$ :}
One can see that
$$
\aligned
&\int_{C_0[0,T]} H(x) m(dx)\\
&=\begin{cases}
h(\vec{0}) & \\
h(\vec{0})+\sum\limits_{l=1}^\infty \frac{1}{(2l)!}\int_{C_0[0,T]} (\langle \alpha_1,x\rangle+\cdots+\langle \alpha_n,x\rangle)^{2l}m(dx)&,  
\end{cases}\\
&=\begin{cases}
h(\vec{0}) & \\
h(\vec{0})+\sum\limits_{l=1}^\infty   I_{2l} & 
\end{cases}\\
\endaligned
$$
for $l=1,2,\cdots$, where
$$
\aligned
I_{2l}&= \sum\limits_{2p_1+\cdots+2p_n=2l}   \frac{1}{(2p_1)!\cdots (2p_n)!}
 \int_{C_0[0,T]} \langle \alpha_1,x\rangle^{2p_1}\cdots\langle \alpha_n,x\rangle^{2p_n} m(dx) \\
&=\sum\limits_{2p_1+\cdots+2p_n=2l}   \frac{1}{(2p_1)!\cdots (2p_n)!}
 \int_{C_0[0,T]} F(x) m(dx),
\endaligned
$$
where $F$ is given by equation \eqref{eq:functional} above.
  This means that we can give the formula for analytic FFT as the series approximation by using equation \eqref{eq:for22}  in Theorem \ref{thm:1}.

{\bf  Step $4$ :} We can conclude that
\begin{equation} \label{eq:series}
T_q^{(1)}(H_r) \rightarrow T_q^{(1)}(H)
\end{equation}
in the sense $L_1(C_0[0,T])$. In fact, for each $\lambda>0$, we have
$$
\aligned
&\int_{C_0[0,T]}|T_\lambda(H)(y)-T_\lambda(H_r)(y)| m(dy)\\
&\leq \int_{C_0[0,T]}\int_{C_0[0,T]} |H(\lambda^{-\frac12}x+y)-H_r(\lambda^{-\frac12}x+y)| m(dx) m(dy)\\
&= \int_{C_0[0,T]} |H(\sqrt{\lambda^{-1}+1}z)-H_r(\sqrt{\lambda^{-1}+1}z)| m(dz)  \\
&\rightarrow 0
\endaligned
$$ 
as $r\rightarrow\infty$. The equality is obtained from the condition \eqref{eq:concon} and the Fubini theorem for the Wiener integrals. Also, by using the uniqueness of the analytic extension and the limit, we obtain equation \eqref{eq:series} as desired.
Hence the
 series approximation of the analytic FFT of functional $H$ is given by the formula
$$
 T_q^{(1)}(H)(y)=\lim_{r\rightarrow \infty}  T_q^{(1)}(H_r)(y)
$$
in the sense $L_1(C_0[0,T])$,
where
$$
\aligned
&T_q^{(1)}(H_r)(y)\\
&=\begin{cases}
h(\vec{0}) &,r \,\,\hbox{is odd}\\
h(\vec{0})+\sum\limits_{l=1}^r \biggl(\sum\limits_{2p_1+\cdots+2p_n=2l} \frac{1}{(2p_1)!\cdots (2p_n)!}\\
\qquad\qquad\qquad\qquad \times \prod_{j=1}^l \biggl[ \sum\limits_{s=0}^{p_j} {}_{2p_j} C_{2s} \biggl(\frac{i}{q}\biggr)^s \langle \alpha_j,y \rangle^{2p_j-2s}     {(2s-1)!!}  \biggr] \biggr)  &, r \,\,\hbox{is even}
\end{cases}\\
\endaligned
$$
for s-a.e. $y\in C_0[0,T]$.
 
 We finish this paper by giving a remark.

\begin{remark}
In order to establish the series approximation with respect to the analytic FFT, we assumed equation \eqref{eq:concon} above. There are many functions so that \eqref{eq:concon} holds. For examples, the all polynomial functions on $\mathbb{R}^n$, exponential functions
$\exp\{a\sum\limits_{j=1}^n u_j\}$ for real number $a$ and the trigonometric function $\sin(\sum\limits_{j=1}^n u_j), \cos(\sum\limits_{j=1}^n u_j^2)$ and so on. We can get many formulas for the analytic FFTs as the series approximation.
\end{remark}


\end{document}